\documentclass[a4paper,11pt]{article}
\usepackage{amsfonts,amsmath,amssymb,amsthm}
\usepackage{color}

\newtheorem{lemma}{Lemma}[section]
\newtheorem{theorem}[lemma]{Theorem}
\newtheorem{corollary}[lemma]{Corollary}
\newtheorem{proposition}[lemma]{Proposition}
\newtheorem{remark}[lemma]{Remark}
\newtheorem{definition}[lemma]{Definition}

\def\Om{\Omega}
\def\R{\mathbb{R}}

\def\E{\widetilde{E}}
\def\e{\varepsilon}

\def\wE{\widetilde{E}}

\def\Div{\textup{div}\,}

\def\x{{\bar x}}
\def\t{{\bar t}}

\begin{document}

\title{Mean curvature flow with obstacles}

\author{
        L. Almeida
        \thanks{CNRS, UMR 7598, Laboratoire Jacques-Louis Lions, France, \texttt{luis@ann.jussieu.fr}} 
        \thanks{UPMC Univ Paris 06, UMR 7598, Laboratoire Jacques-Louis Lions, France},
        \and A. Chambolle
        \thanks{CMAP, Ecole Polytechnique, CNRS, France,
        \texttt{antonin.chambolle@cmap.polytechnique.fr}},
        \and M. Novaga
        \thanks{Dip. di Matematica, Universit\`a di Padova,
        via Trieste 63, 35121 Padova, Italy,
        \texttt{novaga@math.unipd.it}}
}
\date{}
\maketitle

\begin{abstract}
We consider the evolution of fronts by mean curvature in the presence of obstacles.
We construct a weak solution to the flow
by means of a variational method, corresponding to an implicit time-discretization scheme. 
Assuming the regularity of the obstacles, in the two-dimensional case 
we show existence and uniqueness of a regular solution before the onset of singularities. 
Finally, we discuss an application of this result to the positive mean curvature flow.
\end{abstract}


\section{Introduction}

Motivated by several models in physics, biology and material science, 
there has been a growing interest in recent years
towards the rigorous analysis of front propagation in heterogeneous media, 
see \cite{Phillips,BCN,CB,DKY,CLS} and references therein.
In this paper, we analyze the evolution by mean curvature of an interface 
in presence of hard obstacles which can stop the motion. 
Even if this is a prototypical model of energy driven front propagation in a medium with obstacles,
to our knowledge there are no rigorous results concerning existence, uniqueness and regularity of the flow. 
On the other hand, we mention that the corresponding stationary problem, the so-called {\em obstacle problem}, 
has been studied in great detail, see \cite{Mir, Caf} and references therein.

\smallskip

To be more precise, given an open set $\Omega\subset\R^n$, 
we consider the evolution of a hypersurface $\partial E(t)$, with the constraint  
$E(t)\subset \Om$ for all $t\ge 0$, where $\Omega$ is an open subset of $\R^n$ and $\R^n\setminus \Om$
represents the obstacles. The corresponding geometric equation formally reads 
(we refer to Section \ref{secobst} for a precise definition):
\begin{equation}\label{formalevol}
v(x) = \left\{\begin{array}{ll}
\kappa(x) & {\rm if\ }x\in\Om
\\
\max(\kappa(x),0) & {\rm if\ }x\in\partial\Om
\end{array}\right.
\end{equation}
where $v$ and $\kappa$ denote respectively the normal inward velocity and 
the mean curvature of $\partial E(t)$ .
Notice that the right-hand side of \eqref{formalevol} is discontinuous on $\partial\Om$,
so that the classical viscosity theory \cite{users} does not apply to this case 
(see however \cite{DL,BDL} for a possible approach in this direction).

We are particularly interested in existence and uniqueness of smooth (that is $C^{1,1}$) solutions to \eqref{formalevol}. 
We tackle this problem by means of a variational method first introduced in \cite{ATW,LuckhausSturz} (see also~\cite{AmbrosioMM} for a simpler
description of the same approach),
which is based on an implicit time-discretization scheme for \eqref{formalevol}.

After showing the consistency of the scheme with regular solutions (Theorem \ref{thcon}), we obtain 
a comparison principle and uniqueness of smooth solutions in any dimensions (Corollary \ref{corcomp}).
Moreover, in the two-dimensional case we are also able to prove local in time existence of solutions (Theorem \ref{thshort}). 
Notice that in general one cannot expect existence of regular solutions for all time, due to the presence of singularities of the flow (even in dimension~2). 
On the other hand, due to the presence of the obstacles,
regular solutions do not necessarily vanish in finite time
and may exist for all times. Eventually, we apply our result to the positive curvature flow 
in two dimensions, obtaining a short time existence and uniqueness result (Corollary \ref{corpos}) for $C^{1,1}$-regular flows. 
Indeed, such evolution can be seen as a curvature flow where the obstacle is given by the complementary of the initial set. 

We point out that the study of the positive curvature flow
in Section~\ref{secpos} is related to
some biological models which originally motivated our work:
in several recent studies of actomyosin cable contraction in morphogenesis and tissue repair there is increasing evidence that the contractile structure forms only in the positive curvature part of the boundary curve (see \cite{AlmDem,AlmBagHab} and references therein). Since the contraction of such actomyosin structures can be associated with curvature terms (see \cite{Hutson,Alm2009,Alm2011}), this leads very naturally to consider the positive curvature flow problem. 

Notice that a set evolving according to this law is always nonincreasing with respect to inclusion,
which is a feature not satisfied by the usual curvature flow. This shows why assembling the contractile structure only in the positive curvature portion of the boundary (instead of all around) and thus doing positive
curvature flow (instead of usual curvature flow) is an interesting way to evolve from the biological
point of view: it corresponds to making our wound (or hole) close in a
manner where we never abandon any portion of the surface we have already
managed to cover since we started closing.

We also remark that the positive curvature flow is useful
in the context of image analysis \cite[p. 204]{Sethian},
and appears naturally in some differential games~\cite{KohnSerfaty}.

\section{Notation}

Given an open set $A\subseteq\R^n$, a function $u\in
L^1(A)$ whose distributional gradient $Du$ is
a Radon measure with finite total variation in $A$ is called a function of bounded variation,
and the space of such functions will be denoted by $BV(A)$. The total
variation of $Du$ on $A$ turns out to be
\begin{equation}
\label{vartot}
\sup \left\{\int_A u~ {\Div z}~dx : z \in {C^\infty_0(A; \R^n),}~\vert z(x)\vert \leq 1~\forall x\in A \right\},
\end{equation}
and will be denoted by $\vert Du\vert (A)$ or by $\int_A \vert Du\vert$. 
The map $u \to\vert D u\vert(A)$ is $L^1(A)$-lower
semicontinuous, and $BV(A)$ is a Banach
space when endowed with the norm $\Vert u\Vert : = \int_A \vert u \vert \,dx +
\vert Du\vert (A)$. We refer to \cite{AFP} for a comprehensive treatment of the subject.

We say that a set $E$ satisfies the exterior (resp. interior) $R$-ball condition, for some $R>0$, 
if for any $x\in\partial E$ there exists a ball $B_R(x')$, 
with $x\in \partial B_R(x')$ and $B_R(x')\cap E=\emptyset$ (resp. $B_R(x')\subseteq E$). 
Notice that a set $E$ with compact boundary satisfies both the interior
and the exterior $R$-ball condition, for some $R>0$, if and only if $\partial E$ is of class $C^{1,1}$.

\section{The implicit scheme}\label{SecSch}

Following the celebrated papers~\cite{ATW,LuckhausSturz}, 
we shall define an implicit time discrete scheme for \eqref{formalevol}.
As a preliminary step,
we consider solutions of the Total Variation minimization problem
with obstacles; the scheme is then defined in Definition~\ref{defscheme} below. 

Let $B\subset\R^n$ be an open set and let $v:B\to [-\infty,\infty)$
be a measurable function, with $v^+\in L^2(B)$.
Following \cite{ATW,ChamMCM,LuckhausSturz}, 
given $h>0$ and $f \in L^2(B)$, we let $S_{h,v}(f,B)\in L^2(B)\cap BV(B)$ be the unique minimizer of the problem 
\begin{equation}\label{eqfob}
\min_{u\ge v}\int_{B} |Du| + \frac{1}{2h}\int_{B}(u-f)^2\,dx.
\end{equation}

We have the following comparison result (see \cite[Lemma 2.1]{ChamMCM}).

\begin{proposition}\label{promono}
The operator
$S_{h,\cdot}(\cdot,B)$ is monotone, in the sense that $u_1=S_{h,v_1}(f_1,B)
\ge u_2=S_{h,v_2}(f_2,B)$ whenever $f_1\ge f_2$ and $v_1\ge v_2$~a.e.
\end{proposition}
\begin{proof}
The idea is simply to compare the sum of the energies of $u_1$ and $u_2$,
with the sum of the energy of $u_1\wedge u_2$ (which is admissible in
the problem defining $u_2$) and of $u_1\vee u_2$ (which is admissible
in the problem defining $u_1$). The conclusion follows from the
uniqueness of the solution to~\eqref{eqfob}.
\end{proof}

\begin{proposition}\label{probound}
Assume $f, v^+\in L^\infty(B)$: then $u=S_{h,v}(f,B)\in L^\infty(B)$ and 
\[
\|S_{h,v}(f,B)\|_{L^\infty(B)}\le \max\left( \|f\|_{L^\infty(B)},
\|v^+\|_{L^\infty(B)}\right).
\]
\end{proposition}

\begin{proof}
Again, the proof is trivial. It is enough check that the energy of
$u_M=(u\vee -M)\wedge M$ is less than the energy of $u$, while
$u_M$ is admissible as soon as $M\ge \max\left( \|f\|_{L^\infty(B)},
\|v^+\|_{L^\infty(B)}\right)$.
\end{proof}

\begin{theorem}\label{teoglob}
Let $v:\R^n\to [-\infty,+\infty)$ be a measurable function with
$v^+\in L^\infty_{\rm loc}(\R^n)$, $f\in L^\infty_{\rm loc}(\R^n)$,
and $h>0$. There exists a unique function 
$u\in L^\infty_{\rm loc}(\R^n)
\cap BV_{\rm loc}(\R^n)$, which we shall denote by $S_{h,v}(f)$,
such that for all $R>0$ and $p\in (n,+\infty)$ there holds
\[
\lim_{M\to\infty}\|u-S_{h,v}(f,B_M)\|_{L^p(B_R)}=0.
\]

This function is characterized by the fact that $u\ge v$ a.e.,
and for any $R$ and any $\varphi\in BV(\R^n)$ with support in $B_R$
and $u+\varphi\ge v$ a.e.,
\begin{equation*}
\int_{B_R} |Du|+ \frac{1}{2h} \int|u-f|^2\,dx \ \le\ 
\int_{B_R} |D(u+\varphi)|+ \frac{1}{2h} \int|u+\varphi-f|^2\,dx\,.
\end{equation*}
\end{theorem}

\begin{proof}
We shall show a bit more: for any $M>0$, let us denote by $u_M$
an arbitrary local minimizer of~\eqref{eqfob}, in the sense that
\begin{equation}\label{locmin}
\int_{B_M} |Du_M|+ \frac{1}{2h} \int|u_M-f|^2\,dx \le
\int_{B_M} |D(u_M+\varphi)|+ \frac{1}{2h} \int|u_M+\varphi-f|^2\,dx
\end{equation}
for any $\varphi\in BV(B_M)$ with compact support. We will show
that $(u_M)_{M\ge 2R}$ is a Cauchy sequence in $L^p(B_R)$, provided $p>n$.

To start, let us consider $\psi:\R\to \R_+$ a smooth, nondecreasing and
bounded function with $0\le \psi(s)\le Cs^+$ for any $s$. Let $M'>M>0$, 
and let $\varphi\in C_c^\infty(B_M;\R_+)$, which we extend by zero to $B_{M'}$.
We denote $u=u_M$, $u'=u_{M'}$. Let $t>0$: observe that
\[
\begin{array}{ll}
u'(x) + t\psi(u(x)-u'(x))\varphi(x)& \ge\  u'(x)\ \ge\ v(x)
\\
\\
u(x) - t\psi(u(x)-u'(x))\varphi(x)&  \ge\ u(x)-tC\sup\varphi\, (u(x)-u'(x))^+
\\
& \ge\  u(x)-(u(x)-u'(x))^+ 
\\ 
& =\ \min\{u(x),u'(x)\}\\ 
&\ge\ v(x)
\end{array}
\]
for almost every $x\in\R^n$, as soon as $t\le (C\sup\varphi)^{-1}$.

Hence, we deduce from~\eqref{locmin} that for $t$ small enough,
\begin{multline*}
\int_{B_M} |D(u-t\psi(u-u')\varphi)|+\frac{1}{2h}\int_{B_M}
|u-t\psi(u-u')\varphi-f|^2\,dx\\ \ge\ 
\int_{B_M} |Du|+\frac{1}{2h}\int_{B_M}|u-f|^2\,dx 
\end{multline*}
and
\begin{multline*}
\int_{B_M} |D(u'+t\psi(u-u')\varphi)|+\frac{1}{2h}\int_{B_M}
|u'+t\psi(u-u')\varphi-f|^2\,dx\\ \ge\ 
\int_{B_M} |Du'|+\frac{1}{2h}\int_{B_M}|u'-f|^2\,dx \,,
\end{multline*}
which we sum to obtain
\begin{multline*}
\frac{t}{h}\int_{B_M} (u-u')\psi(u-u')\varphi\,dx\ 
\le\ \frac{t^2}{h}\int_{B_M} (\psi(u-u')\varphi)^2\,dx\\
+\,\int_{B_M} |Du-t\psi'(u-u')(Du-Du')\varphi-t\psi(u-u')\nabla \varphi|
\\+|Du'+t\psi'(u-u')(Du-Du')\varphi+t\psi(u-u')\nabla \varphi|-|Du|-|Du'|\;.
\end{multline*}
For $\rho\le t\|\varphi\|_\infty\|\psi'\|_\infty\le 1$ and $t$ 
small enough, the integrand in the right-hand side has the form
\begin{multline*}
|p-\rho(p-p')-tq|+|p'+\rho(p-p')+tq|-|p|-|q|
\\ \le\, 2t|q| + (1-\rho)|p|+\rho|p'|+(1-\rho)|p'|+\rho|p|-|p|-|q|
\,=\,2t|q|\,
\end{multline*}
and we obtain
\[
\frac{t}{h}\int_{B_M} (u-u')\psi(u-u')\varphi\,dx
\ \le\ \frac{t^2}{h}\int_{B_M}(\psi(u-u')\varphi)^2\,dx
\,+\, 2t\int_{B_M} \psi(u-u')|\nabla\varphi|\,dx\,.
\]
Dividing by $t$ and letting $t\to 0$, we deduce
\begin{equation}\label{estimdiff}
\int_{B_M} (u-u')\psi(u-u')\varphi\,dx
\ \le\  2h\int_{B_M} \psi(u-u')|\nabla\varphi|\,dx\,.
\end{equation}
Consider now, for $p> 2$, the function $\psi(s)=(s^+)^{p-1}$:
we want to show that~\eqref{estimdiff} still holds. We approximate
$\psi$ with $\psi_k(s) = k\tanh(\psi(s)/k)$, for $k\ge 1$.
The functions $\psi_k$
satisfy the assumptions which allowed us to establish~\eqref{estimdiff},
so that it holds with $\psi$ replaced with $\psi_k$. Moreover,
$\lim_{k\to\infty} \psi_k(u-u')=\sup_{k\ge 1} \psi_k(u-u')=\psi(u-u')$,
and in the same way $\sup_{k\ge 1} (u-u')\psi_k(u-u')=(u-u')\psi(u-u')$.
Hence, the monotone convergence theorem shows that \eqref{estimdiff}
also holds, in the limit, for $\psi$, as claimed.

We can take $\varphi(x)=\varphi_0(|x|/M)^p$, for some
$\varphi_0\in C_c^\infty([0,1);\R_+)$ which is $1$ on $[0,1/2]$.
It follows from~\eqref{estimdiff} and H\"older's inequality that
\begin{multline*}
\int_{B_M}  \left[(u-u')^+\varphi_0(|x|/M)\right]^p\,dx
\\\le\, 
2h \int_{B_M}  \left[(u-u')^+\varphi_0(|x|/M)\right]^{p-1}
\frac{p}{M}\left|\varphi_0'(|x|/M)\right|\,dx
\\
\le\,2h
\left[\int_{B_M}\left[(u-u')^+\varphi_0(|x|/M)\right]^p
\right]^{1-\frac{1}{p}}
\left[\int_{B_M}\left(\frac{p}{M}\right)^p
\left|\varphi_0'(|x|/M)\right|^p
\right]^{\frac{1}{p}}\,.
\end{multline*}
Hence:
\begin{equation*}
\left\|(u-u')^+\varphi_0\left(\frac{|\cdot|}{M}\right)\right\|_{L^p(B_M)}
\ \le\ \frac{2 h p\omega_n^{1/p}}{M^{1-n/p}}\|\varphi'_0\|_\infty
\end{equation*}
with $\omega_n$ the volume of the unit ball. Exchanging the roles of
$u$ and $u'$ in the previous proof, we find that
\begin{equation}
\|u_M-u_{M'}\|_{L^p(B_{M/2})}\ \le\ \frac{2 h p\,\omega_n^{1/p}}{M^{1-n/p}}\|\varphi'_0\|_\infty\,.
\end{equation}
As in particular $u_M$ (or $u_{M'}$) could, in this calculation, 
have been chosen to be
the minimizer $S_{h,v}(f,B_M)$, which is bounded by Proposition~\ref{probound},
we obtain that $u_{M'}\in L^p(B_{M/2})$ (as well as $u_M$). Hence,
choosing $R>0$, we see that $(u_M)_{M\ge 2R}$ defines
a Cauchy sequence in $L^p(B_R)$, provided $p>n$. It follows that it
converges to some limit $u\in L^p(B_R)$. As $R$ is arbitrary, we build
in this way a function $u$ which clearly satisfies the thesis of the theorem.
\end{proof}

\begin{corollary}\label{Scompar}
Assume $f\ge f'$, $v\ge v'$, $h>0$, then $S_{h,v}(f)\ge S_{h,v'}(f')$.
\end{corollary}

\begin{proof}
It follows from Proposition~\ref{promono} and the definition of $S_{h,v}(f)$.
\end{proof}

\begin{corollary}\label{corlip}
If $f,v$ are uniformly continuous on $\R^n$, with a modulus of
continuity $\omega(\cdot)$, 
then $S_{h,v}(f)$ is also uniformly
continuous with the same modulus of continuity.
\end{corollary}

\begin{proof}
It follows from the previous corollary. For $z\in \R^n$,
let $v'(x)\ :=\ v(x-z)-\omega(|z|)\le v(x)$ and
$f'(x)\ :=\ f(x-z)-\omega(|z|)\le f(x)$. Then,
$S_{h,v'}(f')= S_{h,v}(f)(\cdot - z)-\omega(|z|)\le S_{h,v}(f)$, which shows
the corollary.
\end{proof}

Observe that, if $f,v$ are uniformly continuous, then $S_{h,v} (f,B)$ satisfies the elliptic equation
\begin{equation}\label{eulob}
- \Div z + \frac{u-f}{h} = 0 
\qquad {\rm on\ }\{ x\in B:\ u(x)>v(x)\}.
\end{equation}
where the vector field $z$ satisfies 
$|z|=1$ and $z=Du/|Du|$ whenever $|Du|\ne 0$.

\begin{proposition}\label{proset}
Assume that $f(x)\to \infty$ as $|x|\to \infty$, and let $s\in\R$.
Then the set $\{S_{h,v}(f)<s\}$ is the 
minimal solution of the problem
\begin{equation}\label{varmin}
\min_{E\subset \{v<s\}} P(E) + \int_E \frac{f-s}{h}\,dx.
\end{equation}
Similarly, 
the set $\{S_{h,v}(f)\le s\}$ is the 
maximal solution of 
\begin{equation}\label{varmax}
\min_{E\subset \{v\le s\}} P(E) + \int_E \frac{f-s}{h}\,dx.
\end{equation}
\end{proposition}

\begin{proof}
Let $M>0$ and consider the set $E^s_M=\{S_{h,v}(f,B_M)<s\}$.
Reasoning as in \cite{BCCN} (see also \cite[Sec. 2.2.2]{ChFor})
one can show that $E^s_M$ is the minimal solution of 
\[
\min_{E\subset B_M\cap \{v<s\}} P(E,B_M) + \int_E \frac{f-s}{h}\,dx.
\]
Since $f$ is coercive, the sets $E^s_M$ do not depend on $M$ for $M$ big enough, 
and coincide with the set $\{S_{h,v}(f)< s\}$,
so that the result follows letting $M\to +\infty$.

The second assertion regarding the set $\{S_{h,v}(f)\le s\}$ can be proved analogously.
\end{proof}

\section{Mean curvature flow with obstacles}\label{secobst}

Let us give a precise definition of the flow \eqref{formalevol}.
Given a set $E\subset\R^n$ we denote by 
\[
d_E(x) := {\rm dist}(x,E)-{\rm dist}(x,\R^n\setminus E)\qquad x\in\R^n
\]
the signed distance function from $E$, which is negative inside $E$ and positive outside.

\begin{definition}\label{defsubsuper}
Given a family of sets $E(t)$, $t\in [0,T]$, we set
$$d(x,t):= d_{E(t)}(x).$$
We say that $E(t)$ is a $C^{1,1}$ supersolution of \eqref{formalevol}
if there exists a bounded open set $U\subset\R^n$ such that $E(t)\subset\Om$ and
$\partial E(t)\subset U$ for all $t\in [0,T]$, 
\begin{equation}\label{regol}
\begin{aligned}
d&\in {\rm Lip}(U\times [0,T])
\\
|\nabla^2 d|&\in L^\infty(U\times [0,T])
\end{aligned}\end{equation}
and 
\begin{equation}\label{eqdistsup}
\frac{\partial d}{\partial t} \ge \Delta d + O(d)\qquad {\rm a.e.\ in\ }U\times [0,T].
\end{equation}
We say that $E(t)$ is a $C^{1,1}$ subsolution of~\eqref{formalevol}
if \eqref{eqdistsup} is replaced by
\begin{equation}\label{eqdistsub}
\frac{\partial d}{\partial t} \le \Delta d + O(d)\qquad {\rm a.e.\ in\ }
\left(U\times [0,T]\right)\cap \{d>d_\Om\},
\end{equation}
and we say that 
$E(t)$ is a $C^{1,1}$ solution of~\eqref{formalevol}
if it is both a supersolution and a subsolution. 
\end{definition}

We now fix an open set $\Om\subset\R^n$ (representing the complement of the obstacle) and a compact set $E\subseteq\Om$.
The case when $E^c$ is compact can be treated with minor modifications.

Since $E$ is compact, without loss of generality we can assume that $\Om$ is bounded. Indeed, as it will be clear from the sequel,
replacing $\Om$ with $\Om\cap B_M$ will not affect our construction, provided $B_M\supset E$.

\begin{definition}\label{defscheme}
Let $h>0$ and set 
\begin{equation}\label{Th}
T_h E := \{ S_{h,d_\Om} (d_E)< 0\}.
\end{equation}
Given $t>0$, we let
$$
E_h(t):=T^{[t/h]}_hE 
$$
be the discretized evolution 
of $E$ defined by the scheme $T_h$. 
\end{definition}

Notice that $T_h E$ is an open subset of $\Om$ and, by Proposition \ref{proset},  $T_h E$ is the minimal solution
of the geometric problem
\begin{equation}\label{probset}
\min_{F\subseteq \Om} P(F) + \frac{1}{h}\int_{F}d_E\,dx
\end{equation}
or equivalently
\[
\min_{F\subseteq \Om} P(F) + \frac{1}{h}\int_{F\triangle E}|d_E|\,dx.
\]
When $\Om=\R^n$ this corresponds to the implicit scheme 
introduced in \cite{ATW,LuckhausSturz} for the mean curvature flow. Here, 
from \eqref{eulob} it also follows that $T_h E$ satisfies
\begin{equation}\label{eukappa}
\kappa +\frac{d_E}{h}=0
\qquad {\rm on}\quad \partial T_h E\setminus \partial\Om.
\end{equation}

\begin{remark}\label{remalm}\rm
Observe that from Proposition \ref{promono} it follows 
$$
E_1\subset E_2 \Rightarrow 
T_h E_1\subset T_h E_2.
$$
Moreover, by Corollary \ref{Scompar} we have
$S_{h,d_\Om} (d_E)\ge S_{h,-\infty} (d_E)$ which implies
$T_h E \subseteq \widetilde T_h E:=\{ S_{h,-\infty} (d_E)<0\}$.
Notice that $\widetilde T_h E$ is the scheme introduced in \cite{ATW,LuckhausSturz} 
for the (unconstrained) mean curvature flow.
\end{remark}

{}From the general regularity theory for minimizers of the perimeter with 
a smooth obstacle \cite{Mir,Caf} we have the following result.

\begin{proposition}\label{propreg}
Let $\partial \Om$ be of class $C^{1,1}$, $E\subseteq\Om$ and $h>0$.
Then there exists a closed set $\Sigma\subset\partial T_hE\cap \Omega$ such that 
$\mathcal H^s(\Sigma)=0$ for all $s>n-8$,  $\partial T_hE\setminus \Sigma$ is of class $C^{1,1}$,
and $(\partial T_hE\cap \Omega)\setminus \Sigma$ is $C^{2,\alpha}$ for
any $\alpha<1$.
\end{proposition}

\begin{proposition}\label{propres}
Let $\partial \Om$ be of class $C^{1,1}$. Then 
there exists $C(\Om)>0$ such that
\[
T_h E=\{S_{h,-\infty}(d_E+C h \chi^{}_{\Om^c})<0\}
\]
for all $C\ge C(\Om)$. In particular $T_hE$ is a minimizer of the prescribed
curvature problem
\begin{equation}\label{prescribedforcing}
\min_F P(F)+C|F\setminus\Om|+\frac{1}{h}\int_F d_E\,dx.
\end{equation}
\end{proposition}

\begin{proof}
We recall that $S_{h,-\infty}(d_E+C h \chi^{}_{\Om^c})$ is the limit,
 as $M\to\infty$, of the minimizer $u_M$ of the variational problem
\begin{equation}\label{uMforcing}
\min_{u\in BV(B_M)}\int_{B_M}|Du|+\frac{1}{2h}\int_{B_M}(u-d_E-C h \chi^{}_{\Om^c})^2dx
\end{equation}
{}From Proposition \ref{proset} it follows that $T_hE$ is the
minimal solution to \eqref{probset},
while $$\bar F=\{S_{h,-\infty}(d_E+C h \chi^{}_{\Om^c})<0\}$$ is the minimal solution
to~\eqref{prescribedforcing}. If $\bar F\subset\Om$, then $|\bar F\setminus
\Om|=0$ and both $\bar F$ and $T_h E$ solve the same problem, and
they must therefore coincide.

In order to show that $\bar F\subset\Om$, it is enough to find
a positive constant  $\tilde{C}$ such that
for all $x\not\in\Om$, $u_M\ge \tilde{C}>0$ for $M$ large enough.

By assumption, $\Om$ satisfies an exterior $R$-ball condition, for some $R>0$, that is,
for any $x\not\in \Om$, there is a ball
$B_R(x')$ with $x\in B_R(x')$ and $B_R(x')\cap\Om=\emptyset$. 
If $M$ is large enough, we also have $B_R(x')\in B_{M/2}$. Since
$E\subset\Om$, $d_E+hC\chi_{\Om^c}\ge hC\chi_{B_R(x')}$, so that
$u_M$ is larger than the minimizer $u'$ of
\[
\min_{u\in BV(B_M)} \int_{B_M}|Du| +\frac{1}{2h}\int_{B_M} (u-hC\chi_{B_R(x')})^2
\,dx
\]
If $C>n/R$, then it is well known that for $M$ large enough,
$u'\ge (C-n/R)h$ a.e.~in $\chi_{B_R(x')}$~\cite{Meyer}.
The thesis then follows.
\end{proof}


\subsection{Existence of weak solutions}

As a consequence of Proposition \ref{propres}, when 
$\partial \Om$ is of class $C^{1,1}$ the scheme enters the framework
considered in \cite{CN-ATW}. In that case, we can also show
existence of weak solutions in the sense of~\cite{ATW,LuckhausSturz}. 
We observe that the results in~\cite[p.~226]{AmbrosioMM} 
still apply and we can deduce the (approximate)
$1/(n+1)$--H\"older-continuity in time of the discrete flow starting from an initial set $E_0$.
As a consequence, following \cite[Th.~3.3]{AmbrosioMM},
we can pass to the limit, up to a subsequence, and deduce the existence
of a flow $E(t)$, which is H\"older-continuous in time in $L^1(\Om)$.

\begin{theorem}[Existence of H\"older-continuous weak solutions]
Let $\partial \Om$ be of class $C^{1,1}$, let $E\subset\Om$ be a compact set of finite perimeter and such that $|\partial E|=0$. 
Let $E_h(t)$ be the discretized evolutions starting from $E$, defined in Definition \ref {defscheme}. Then there exist a constant 
$C=C(n,E,\Om)>0$, a sequence $h_i\to 0$ and a map $E(t)\to \mathcal P(\Om)$ such that 
\begin{itemize}
\item $E(0)=E$;
\item $E(t)$ is a compact set of finite perimeter for all $t\ge 0$;
\item $\lim_i \vert E_{h_i}(t)\Delta E(t)\vert=0$ for all $t\ge 0$;
\item $\vert E(t)\Delta E(s)\vert\le C |s-t|^\frac{1}{n+1}$ for all $s,t\ge 0$, with $|s-t|\le 1$.
\end{itemize}
\end{theorem}

\subsection{Consistency of the scheme}\label{secst}

The main result of this section (Theorem~\ref{thcon}) is showing 
that the implicit scheme is consistent with regular evolutions,
according to the following definition.

\begin{definition}\label{defconsist}
The scheme $T_h$ is consistent if and only if 
\begin{enumerate}
\item If $E(\cdot)$ is a supersolution (see Def.~\ref{defsubsuper})
in an interval $[t_1,t_2]$,
then for any $t\in [t_1,t_2]$,
any Hausdorff limit of $T_h^n E(t_1)$, $n\to\infty$,
$h\to 0$, $nh\to t-t_1$, contains $E(t)$.
\item If $E(\cdot)$ is a subsolution, this inclusion is reversed.
\end{enumerate}
\end{definition}

\begin{theorem}\label{thcon}
The scheme $T_h$ is consistent.
\end{theorem}

\begin{proof}
The proof consists in building, arbitrarily close to $\partial E(t)$,
strict super and subsolutions of class $C^2$,
of the curvature flow with forcing term $C\chi_{\Omega^c}$, for $C$ large
enough.
Then, the consistency result in \cite[Th. 3.3]{CN-ATW} applies.\smallskip

\noindent{\it Step 1.}
Let $E$ be a subsolution on $[t_1,t_2]$ in the sense of
Definition~\ref{defsubsuper}, let $U\subset\R^n$ be the 
neighborhood associated to $\partial E(t)$ (given by
Definition~\ref{defsubsuper}). Without loss of generality
we can assume $t_1=0$.

Observe that there exists $\rho>0$ such that
$\{|d(\cdot,t)|\leq \rho\}\subset U$ for all $t\in
[0,t_2]$, and the sets $\partial\Om$, $\partial \{d(\cdot,t)\le s\}$, $|s|\le \rho$,
satisfy the interior and exterior $\rho$-ball condition for all times
(in particular $\partial E(t)$ satisfies the condition with radius $2\rho$).

Let $c_\rho\ge (n-1)/\rho^2$, and for $\e>0$ small, let
$$
d_\e (x,t)=d(x,t)-\e - 4 c_\rho\e t \qquad t\in [0,t_2].
$$
Observe that for $\e$ small enough,
$\{|d_\e(\cdot,t)|\le \rho/2\}\subset \{|d(\cdot,t)|\le\rho\}$ for all $t$.
The constant $c_\rho$ is precisely chosen so that in this set, the curvature
of two level surfaces $\{d(\cdot,t)=s\}$ and $\{d(\cdot,t)=s'\}$ at points along
the same normal vector $\nabla d(\cdot,t)$ differ by at most $c_\rho|s-s'|$.

We have, for a.e.~$t\in (0,t_2)$ and $x\in \{|d(\cdot,t)|\le \rho\}\subset U$,
\[
\frac{\partial d_\e}{\partial t}(x,t)\ =\ \frac{\partial d}{\partial t}(\Pi_{\partial E(x,t)}(x),t)
-4 c_\rho\e \,,
\]
thus:
\begin{itemize}
\item If $\Pi_{\partial E(x,t)}(x)\in \Om$, then (by  Definition~\ref{defsubsuper})
\[
\frac{\partial d_\e}{\partial t}(x,t)\ \le \ \Delta d_\e(x,t)-4 c_\rho\e + c_\rho|d|
\ \le\ \Delta d_\e(x,t)+ c_\rho|d_\e| + c_\rho(-4 \e + \e(1+ 4 c_\rho t))
\]
so that if $t\le\bar t= \min(t_2, 1/(2c_\rho))$ and $|d_\e|\le \e/2$,
\begin{equation}\label{ineqin}
\frac{\partial d_\e}{\partial t}(x,t)
\ \le\ \Delta d_\e(x,t)\,-\,c_\rho\frac{\e}{2}.
\end{equation}
\item While if $\Pi_{\partial E(x,t)}(x)\in \partial \Om$, then
$d=d_\Om$ and  almost surely $\partial d/\partial t=0$,
so that $\partial d_\e/\partial t = - 4 c_\rho\e$. On the other hand, there is
a constant $\bar C$ large enough
(of order $1/\rho$, and admissible for Proposition~\ref{propres}) such that
$|\Delta d_\e|\le \bar C$ a.e.~in $\{|d(\cdot,t)|<\rho\}$, and we deduce
\begin{equation}\label{ineqout}
- 4 c_\rho\e \ =\ 
\frac{\partial d_\e}{\partial t}(x,t)\ \le\ \Delta d_\e(x,t)+\bar C - 4 c_\rho\e.
\end{equation}
Moreover, if $d_\e\ge -\e/2$, we have that $d_\Om=d\ge 4c_\rho\e t+\e/2$.
\end{itemize}

Consider a function $g_\e$ which is $\bar C$ in $\{d_\Om\ge \e/2\}$,
$0$ in $\Om$, and smoothly decreasing from $\bar C$ to $0$ as $d_\Om$ decreases
from $\e/2$ to $0$: we deduce from~\eqref{ineqin} and~\eqref{ineqout} that
\[
\frac{\partial d_\e}{\partial t}\ \le\ \Delta d_\e\,+\, g_\e \,-\,c_\rho\frac{\e}{2}
\]
a.e.~in $\{(x,t)\,:\, |d_\e(x,t)|\le \e/2\,, t\in (0,\bar t)\}$.
We have built a strict subflow, as close as we want from $\partial E(t)$,
for  $t\in [0,\bar t]$. The fact that $\bar t$ could be less than $t_2$ is
not an issue, as we will see in the end of the next step.
On the other hand, the consistency
result in~\cite{CN-ATW} requires that $d$ is at least $C^2$ in space, which is
not the case here (and the proof does not extend to $C^{1,1}$ regularity).
For this, we need an additional smoothing of the surface,
which we perform in a second step.\smallskip

\noindent{\it Step 2.}
Now consider a spatial mollifier $\varphi_\eta(x)=\eta^{-n}\varphi(x/\eta)$, with $\eta<<\e$.
For all time let $d^\eta_\e=\varphi_\eta*d_\e$, which is still Lipschitz in $t$ and
now, smooth in $x$. If $\eta$ is small enough, and since $g_\e$ is continuous, we have
\[
\frac{\partial d^\eta_\e}{\partial t}\ \le\ \Delta d^\eta_\e\,+\, g_\e \,-\,c_\rho\frac{\e}{4}
\]
for a.e.~$x,t$ with $|d_\e(x,t)|\le \e/2-\eta$. We can rewrite this equation as
a curvature motion equation with some error term, as follows:
\begin{equation}\label{mcfwitherror}
\frac{\partial d^\eta_\e}{\partial t}
\ \le\  |\nabla d^\eta_\e|\left(\Div \frac{\nabla d^\eta_\e}{|\nabla d^\eta_\e|}
\,+\, g_\e\right)  \,-\,c_\rho\frac{\e}{4} 
\,+\, g_\e(1-|\nabla d^\eta_\e|) \,+\, \frac{(D^2d^\eta_\e\,\nabla d^\eta_\e)\cdot
\nabla d^\eta_\e}{|\nabla d^\eta_\e|^2}\,.
\end{equation}
Now, we have that
\begin{equation}\label{error0}
1\ \ge |\nabla d^\eta_\e|\ge 1-c\eta
\end{equation}
almost everywhere, 
for some constant $c>0$, of order $1/\rho$. Hence, if $\eta$ is small enough, we have
\begin{equation}\label{error1}
g_\e(1-|\nabla d^\eta_\e|)\le c_\rho\e/16.
\end{equation}

We claim that the following estimates holds: there
exists a constant $c>0$ (of order $1/\rho^2$) such that
\begin{equation}\label{error2}
|D^2 d^\eta_\e \,\nabla d^\eta_\e|\ \le\ c \eta\,.
\end{equation}
This will be shown later on (see \textit{Step 3}).
Using \eqref{error0} and \eqref{error2}, we find that
\[
\frac{(D^2d^\eta_\e\,\nabla d^\eta_\e)\cdot
\nabla d^\eta_\e}{|\nabla d^\eta_\e|^2}\ \le\ c_\rho\e/16
\]
if $\eta$ is small enough. Thus \eqref{mcfwitherror} becomes, using \eqref{error1},
\begin{equation}\label{strictsub}
\frac{\partial d^\eta_\e}{\partial t}
\ \le\  |\nabla d^\eta_\e|\left(\Div \frac{\nabla d^\eta_\e}{|\nabla d^\eta_\e|}
\,+\, g_\e\right)  \,-\,c_\rho\frac{\e}{8} \,.
\end{equation}

Since $|D^2 d_\e|\le 1/\rho$ for a.e.~$t$ and $x$ with $|d_\e(x,t)|\le \e/2$, this
is also true for $|D^2 d^\eta_\e|$
(for $|d_\e(x,t)|\le \e/2-\eta$), and using~\eqref{error0}
we can easily deduce that the boundaries of the level sets $E_\e(t)= \{d^\eta_\e(\cdot,t)\le 0\}$
have an
interior and an exterior ball condition with radius $\rho/2$. Together with~\eqref{strictsub},
and using $g_\e\le \bar C\chi_{\Om^c}$, we find that $E_\e(t)$, $0\le t\le \bar t$, is
a strict subflow for the motion with normal speed
 $V=-\kappa- \bar C\chi_{\Om^c}$, and
\cite[Th.~3.3]{CN-ATW} holds. We deduce that there exists $h_0>0$ such
that if $h<h_0$, $\overline{T}_h(E_\e(t)) \subseteq E_\e(t+h)$ for any $t\in [0,\bar t-h]$,
where $\overline{T}_h$ is the evolution scheme defined by
\[
\overline{T}_h E \ =\ \left\{ S_{h,-\infty} (d_E+\bar C h\chi_{\Om^c})\ < \ 0\right\}
\]
for any bounded set $E$. (It corresponds to the time-discretization of the mean
curvature flow with discontinuous forcing term $-C\chi_{\Om^c}$.)
Recall that
if $E\subset \Om$, Proposition~\ref{propres} shows that $\overline{T}_h E=T_h E\subset \Om$.
In particular, for the subflow $E(\cdot)$ considered here, he have
$T^n_h(E(0))=\overline{T}^n_h E(0)$, for all $n$ and $h>0$.
By induction, it follows that as long as $nh\le \bar t$,
\[
T^n_hE(0) \,=\,\overline{T}^n_h E(0)\,\subseteq\, E_\e(nh),
\]
hence $T^{\lfloor t/h \rfloor}_h E(0)$ is in a $3\e$-neighborhood of $E(t)$. Since $\bar t$ only depends
on $\rho>0$ (the regularity of the subflow $E(\cdot)$), we can split $[0,t_2]$ into
a finite number of intervals of size at most $\bar t$ and reproduce this construction
on each interval, making sure that the $\e$ parameter of each interval is less
than one third of the $\e$ of the next interval.

We deduce that for any $\delta>0$, if $h>0$ is small enough,
 then $T^n_hE(0) \subset \{d_{E(nh)}\le \delta\}$,
for $0\le nh \le t_2$. This shows the consistency of $T_h$
with subflows, assuming~\eqref{error2} holds.\smallskip

\noindent{\it Step 3: Proof of estimate~\eqref{error2}.}
Recall that since $d_\e$ is a distance function, $|\nabla d_\e|=1$
almost everywhere. Now, let us compute, for $\eta>0$ small and $x,y\in\{d(\cdot,t)\le
\e/2-\eta\}$:
\begin{multline}\label{eqlem1}
|\nabla d^\eta_\e(x,t)|^2-|\nabla d^\eta_\e(y,t)|^2
\ =\ \left(\nabla d^\eta_\e(x,t)-\nabla d^\eta_\e(y,t)\right)
\cdot \left(\nabla d^\eta_\e(x,t)+\nabla d^\eta_\e(y,t)\right)
\\
=\ \int_{B_\eta}\int_{B_\eta}
\left(\nabla d_\e(x-z,t)-\nabla d_\e(y-z,t)\right)
\cdot \left(\nabla d_\e(x-z',t)+\nabla d_\e(y-z',t)\right)\varphi_\eta(z)\varphi_\eta(z')\,dz\,dz'\,.
\end{multline}

As $|D^2 d_\e|\le 1/\rho$, $\nabla d_\e(\cdot,t)$ is $1/\rho$-Lipschitz,
using $|\nabla d_\e(x-z,t)|^2-|\nabla d_\e(y-z,t)|^2=0$ it follows
\begin{multline*}
\left(\nabla d_\e(x-z,t)-\nabla d_\e(y-z,t)\right)
\cdot \left(\nabla d_\e(x-z',t)+\nabla d_\e(y-z',t)\right)
\\ \le\ |\nabla d_\e(x-z,t)-\nabla d_\e(y-z,t)|\,\frac{2}{\rho}|z-z'|
\ \le\ \frac{2}{\rho^2} |x-y||z-z'|
\end{multline*}
and it follows from \eqref{eqlem1} that
\[
|\nabla d^\eta_\e(x,t)|^2-|\nabla d^\eta_\e(y,t)|^2
\ \le\ \frac{4}{\rho^2}|x-y|\eta\,.
\]
We deduce (letting $y\to x$) that
\[
2|D^2 d^\eta_\e(x,t)\nabla d^\eta_\e(x,t)|\ \le\ \frac{4}{\rho^2}\eta\,,
\]
which is estimate~\eqref{error2}.
\smallskip

\noindent{\it Step 4.} Consistency with superflows: the proof is almost identical
(reversing the signs and inequalities), but simpler for superflows. Indeed,
all the sets we now consider stay in $\Om$ and we do not need to take into
account the constraint or the forcing term $\bar C\chi_{\Om^c}$. 

\end{proof}


We can define a generalized flow as limit of the 
scheme $T_h$ as $h\to 0$.
Given an initial set $E\subseteq\Om$, for all $t\ge 0$ we let 
\begin{equation}\label{eh}
E_h(t) = T_h^{[t/h]}E \qquad {\rm and}\qquad  
E_h =\bigcup_{t\ge 0}E_h(t)\times \{t\}\subset\R^n\times [0,+\infty).
\end{equation}
Then there exists a sequence $(h_k)_{k\geq 1}$ such that
both $E_{h_k}$ and $\R^n\times [0,+\infty)\setminus E_{h_k}={}^c E_{h_k}$ 
converge in the Hausdorff distance (locally in time) 
to $E^*$ and ${}^c E_*$ respectively.

{}From Corollary \ref{Scompar} and Theorem \ref{thcon} we obtain a
comparison and uniqueness result for solutions of \eqref{formalevol}.

\begin{corollary}\label{corcomp}
Let $E_1(t)$ and $E_2(t)$ be respectively a sub- and a supersolution of \eqref{formalevol} for $t\in [0,T]$,
in the sense of Definition \ref{defsubsuper}. Then, if $E_1(0)\subseteq E_2(0)$, it follows that
$E_1(t)\subseteq E_2(t)$ for all $t\in [0,T]$.
In particular, if $\partial E$ is compact and of class $C^{1,1}$, there exists at most one solution $E(t)$ starting from $E$.
Moreover, by Remark \ref{remalm}, $E(t)$ is contained in the solution to the (unconstrained) mean curvature flow starting from $E$.
\end{corollary}

\section{Short time existence and uniqueness in
dimension two}\label{sec2D}

In this section we assume $n=2$ and $\partial\Om$ of class $C^{1,1}$.
In the bidimensional case, the mean curvature is the same as
the total curvature of the boundary $\partial E$. Hence, any estimate
on the mean curvature yields a global estimate on the regularity
of $E$. This will be the key of our construction, for showing the
existence of regular ($C^{1,1}$) solutions to the mean curvature flow
with obstacles.
In higher dimension, this is not true anymore, and showing
the existence of such solutions remains an open problem.


The following result follows as in \cite[Lemma 7]{BCCN}.

\begin{lemma}\label{lemcha}
Let $h>0$ and let $E\subseteq\Om$ with $\partial E$ of class $C^{1,1}$. 
Let $\delta_E$ be the maximum $\delta>0$ such that both 
$\partial E$ and $\partial\Om$ satisfy the $\delta$-ball condition,
and let $u=S_{h,d_\Om}(d_E)$.
Then, for all $\delta'\in (0,\delta_E)$ we have 
\begin{equation}\label{dlip}
|u- d_E| \le\frac{h}{\delta_E-\delta'} \qquad {\rm in\ }\{|d_E|\le\delta'\}
\end{equation}
for all $h<(\delta_E-\delta')^2/3$.
\end{lemma}

\begin{lemma}\label{lemball}
Let $E\subseteq\Om$ with $\partial E$ of class $C^{1,1}$. 
Then, there exists $\delta>0$ and $T>0$ such that 
\begin{equation}\label{eqk}
\partial E_h(t)\ \textrm{satisfies the $\delta$-ball condition for all}\ t\in [0,T].
\end{equation}
\end{lemma}

\begin{proof}
Let $\delta_{E}$ be as in Lemma \ref{lemcha}, and let $K=2/\delta_E$. By Lemma \ref{lemcha}, applied with 
$\delta'=Kh$, we get
\[
{\rm d}_\mathcal{H}(\partial T_h E,\partial E)\le \frac{h}{\delta_E-Kh}
\le \frac{h}{\delta_E}\left( 1+\frac{K}{\delta_E}h+\widehat C\frac{K^2}{\delta_E^2}h^2\right)
\]
for all $h\le h_0:=\delta_E^2/12$, where the constant $\widehat C>0$ is independent of $E$.
Recalling \eqref{eukappa} and Proposition \ref{propreg}, we get
\[
\|\kappa\|_{L^\infty(\partial T_hE)} \le \frac{1}{\delta_E}
\left( 1+\frac{K}{\delta_E}h+\widehat C\frac{K^2}{\delta_E^2}h^2\right)
\]
which implies
\begin{eqnarray}\label{eqiter}
\delta_{T_hE}&\ge& \min\left( \frac{1}{\|\kappa\|_{L^\infty(\partial T_hE)}},
\delta_E-{\rm d}_\mathcal{H}(\partial T_h E,\partial E)\right)
\\ \nonumber 
&\ge& \delta_E \cdot \min\left( 1-\frac{h}{\delta_E^2}\left(1+\frac{K}{\delta_E}h+\widehat C\frac{K^2}{\delta_E^2}h^2\right),
\left( 1+\frac{K}{\delta_E}h+\widehat C\frac{K^2}{\delta_E^2}h^2\right)^{-1}
\right)
\end{eqnarray}
for all $h\le h_0$. 
By iterating \eqref{eqiter} we obtain \eqref{eqk}.
\end{proof}

We now prove a short time existence and uniqueness result for solutions to \eqref{formalevol}.

\begin{theorem}\label{thshort}
Let $\partial\Om$ be of class $C^{1,1}$ and let $E\subseteq\Om$ with $\partial E$ of class $C^{1,1}$. 
Then there exists $T>0$ such that \eqref{formalevol} admits a unique $C^{1,1}$ solution 
$E(t)$ on $[0,T]$ with $E(0)=E$. 
\end{theorem}

\begin{proof}
Let $E_h$ be as in \eqref{eh} and let 
\[
d_h(t)=\left( 1+\left[\frac t h\right]-\frac t h\right)d_{E_h(t)}+
\left( \frac t h-\left[\frac t h\right]\right)d_{E_h(t+h)}\,.
\]
By Lemmas \ref{lemcha} and \ref{lemball} there exist an open set $U\subset \R^n$ and $T>0$ such that 
$\partial E_h(t)\subset U$ for all $t\in [0,T]$ and $|\nabla^2 d_h|\in L^\infty(U\times [0,T])$;
moreover, recalling \eqref{dlip} we also have $d_h\in {\rm Lip}(U\times [0,T])$.
By the Arzel\`a-Ascoli Theorem the functions $d_h$ converge uniformly in $U\times [0,T]$, up to a subsequence as $h\to 0$,
to a limit function $d\in {\rm Lip}(U\times [0,T])$ such that $|\nabla^2 d|\in L^\infty(U\times [0,T])$
and $|\nabla d|=1$ in $U\times [0,T]$.
Letting $E(t)=\{ x:d(x,t)<0\}$, for all $t\in [0,T]$ we then have $E(0)=E$,
$E(t)\subset\Om$ and $\partial E(t)$ is of class $C^{1,1}$.

It remains to show that \eqref{eqdistsup} and \eqref{eqdistsub} hold in $U\times [0,T]$.
{}From Theorem \ref{thcon} it follows that, given a supersolution $\widetilde E(t)$ on $[t_1,t_2]\subset [0,T]$ 
with $\widetilde E(t_1)\subseteq E(t_1)$, we have $\widetilde E(t)\subseteq E(t)$ for all $t\in [t_1,t_2]$,
and the same holds with reversed inclusions if $\widetilde E(t)$ is a subsolution.
This implies that
\[
\frac{\partial d}{\partial t} = \Delta d \qquad {\rm a.e.\ in\ }
\left(U\times [0,T]\right)\cap \{d>d_\Om\}\cap \{d=0\},
\] 
which proves \eqref{eqdistsub}.
Observe that, by parabolic regularity,
$\partial E(t)\cap \Om$ is an analytic curve 
and the equality holds everywhere. 

As we have
\[
\frac{\partial d}{\partial t} = 0 \qquad {\rm a.e.\ in\ }
\left(U\times [0,T]\right)\cap \{d=d_\Om\},
\] 
the proof of \eqref{eqdistsup} amounts to show 
\begin{equation}\label{eqdelta}
\Delta d \le 0 \qquad {\rm a.e.\ in\ }
\left(U\times [0,T]\right)\cap \{d=d_\Om\}.
\end{equation}
Assume by contradiction that there exist $(\x,\t)\in (U\times (0,T))\cap \{d=d_\Om\}$ such that 
\begin{equation}\label{eqalpha}
\frac{\partial d}{\partial t}(\x,\t) = 0<\Delta d(\x,\t)=\Delta d_\Om(\x).
\end{equation}
Without loss of generality we can assume $d(\x,\t)=d_\Om(\x)=0$, and $d_\Om$ is
twice differentiable (in the classical sense) at $\x$.

Let us take an open set $\widetilde\Om\supset\Om$ with (compact) boundary of class $C^\infty$ and such that 
$$
\x\in\partial\widetilde\Om \qquad {\rm and}\qquad
\Delta d_{\widetilde\Om}(\x)\ge \Delta d_\Om(\x)>0. 
$$
We let $\widetilde\Om(t)$, for $t\in [0,\tau]$ and $\tau>0$, be the evolution by curvature of $\widetilde\Om$ 
\cite{ATW}, and observe that $\widehat E(t)=\widetilde\Om(t-\bar t)$, $t\in [\bar t,\bar t+\tau]$, 
is a subsolution in the sense of Definition \ref{defsubsuper}.
In particular, by Theorem \ref{thcon}
\[
E(t)\subseteq \widehat E(t) \qquad {\rm for\ all\ }t\in [\bar t,\bar t+\tau],
\]
but this implies, letting $\hat d(x,t)=d_{\widehat E(t)}(x)$ and recalling \eqref{eqalpha}, 
\[
0=\frac{\partial d}{\partial t}(\bar x,\bar t)\ge \frac{\partial \hat d}{\partial t}(\bar x,\bar t)
= \Delta d_{\widetilde\Om}(\x)\ge \Delta d_\Om(\x)>0,
\]
leading to a contradiction. This proves \eqref{eqdelta} and thus \eqref{eqdistsup}.

Finally, the uniqueness of $E(t)$ follows from Corollary \ref{corcomp}.
\end{proof}

\begin{remark}\rm
Notice that in Theorem \ref{thshort} it is enough to assume that $\Om$ satisfies the 
exterior $R$-ball condition for some $R>0$, which is a weaker assumption than 
requiring $\partial \Om$ to be of class $C^{1,1}$. Indeed, we can approximate $\Om$ with the sets
\[
\Om_\rho := \bigcup_{B_\rho(x)\subseteq\Om}B_\rho(x) \qquad \rho>0.
\]
Notice that $\Omega_\rho\subseteq \Om$ and $\partial \Om_\rho$ is of class $C^{1,1}$, for all $\rho>0$.
If we take $\rho$ small enough so that $E\subseteq\Om_\rho$ then,  
by Theorem \ref{thshort} applied with $\Om$ replaced by $\Om_\rho$, we obtain a solution 
$E_\rho(t)$ on $[0,T_\rho]$. However, $E_\rho(t)$ is also a solution 
of the original problem, with constraint $\Om$ instead of $\Omega_\rho$, since
$\Omega_\rho$ is a subsolution to \eqref{formalevol} in the sense of Definition \ref{defsubsuper}.
\end{remark}

\section{Positive mean curvature flow}\label{secpos}

In this section we consider the geometric equation
\begin{equation}\label{eqkpos}
v = \max(\kappa, 0).
\end{equation}
Notice that, by passing to the complementary set, \eqref{eqkpos} 
includes the evolution by negative mean curvature $v=\min(\kappa,0)$.

\begin{definition}\label{defsubpos}
Given a family of sets $E(t)$, $t\in [0,T]$, we set
$$d(x,t):= d_{E(t)}(x).$$
We say that $E(t)$ is a $C^{1,1}$ solution of \eqref{eqkpos}
if there exists a bounded open set $U\subset\R^n$ 
such that $\partial E(t)\subset U$ for all $t\in [0,T]$, 
\begin{equation*}
d\in {\rm Lip}(U\times [0,T])
\qquad \qquad 
|\nabla^2 d|\in L^\infty(U\times [0,T])
\end{equation*}
and 
\begin{equation}\label{eqdistpos}
\frac{\partial d}{\partial t} = \max\left(\Delta d, 0\right) + O(d)\qquad {\rm a.e.\ in\ }U\times [0,T].
\end{equation} 
\end{definition}

\begin{lemma}\label{lemcon}
Let $E_1(t)$ and $E_2(t)$, with $t\in [t_1,t_2]$, 
be two $C^{1,1}$ solutions of \eqref{eqkpos} 
in the sense of Definition \ref{defsubpos}. 
Then, if $E_1(t_1)\subseteq E_2(t_1)$, it follows that
$E_1(t)\subseteq E_2(t)$ for all $t\in [t_1,t_2]$.
In particular, if $\partial E$ is compact and of class $C^{1,1}$, there exists at most one solution $E(t)$ 
starting from $E$.
\end{lemma}

\begin{proof}
Notice that it is enough to prove the thesis with $t_2=t_1+\tau$, for some $\tau>0$,
since the general claim then follows by iteration.
Fix $\e>0$, let $d_\e(x,t):= d_{E_2(t)}(x)+\e+C\e t$ and let 
\[
E_\e(t) := \{ x:\,d_\e(x,t)\le 0\} \qquad t\in [t_1,t_1+\tau],
\] 
where the positive constants $C,\tau$ will be determined later. 
Notice that $\partial E_\e$ is compact and of class $C^{1,1}$ for all $\e$ small enough,
and $E_\e(t)\to E_2(t)$ as $\e\to 0$. A direct computation gives
\begin{equation}\label{eqee}
\frac{\partial d_\e}{\partial t} \ge \max\left(\Delta d_\e+ \e \left( C - CK^2\tau - K^2\right), 0\right) 
+O(d_\e)
\qquad {\rm a.e.\ in\ }U\times [t_1,t_1+\tau],
\end{equation}
where 
\[
K= \sup_{x\in [t_1,t_2]}\|\Delta d_{E_2(t)}\|_{L^\infty(\partial E_2(t))}.
\]
If we choose $C=2K^2$ and $\tau<1/C$, \eqref{eqee} implies that $E_\e(t)$ is a supersolution 
of $\eqref{eqkpos}$. 
Letting $D_\e(t):= {\rm dist}(\partial E_1(t),\partial E_\e(t))$, we have that
$D_\e$ is Lipschitz continuous, $D_\e(0)\ge \e$ and $D_\e'(t)\ge 0$ for a.e. $t\in [t_1,t_1+\tau]$. 
As a consequence, $E_\e(t)\subseteq E_1(t)$ for all $t\in [t_1,t_1+\tau]$, and the thesis follows by letting $\e\to 0$.
\end{proof}

\begin{remark}\rm
Notice that the viscosity theory \cite{users} applies to equation \eqref{eqkpos},
since the function $\kappa\to\max(\kappa, 0)$ is continuous. Then,
Lemma \ref{lemcon} implies that, if the initial set has compact boundary 
of class $C^{1,1}$, the corresponding viscosity solution does not create 
fattening, i.e. is unique, before the onset of singularities.
Corollary~\ref{corpos} below will establish the existence of
such $C^{1,1}$ solutions.
\end{remark}

Given $E\subset\R^n$ and $h>0$, we set $E^0_h=\wE^0_h= E$ and, by iteration,
\begin{equation}\label{eqiterbis}
\begin{aligned}
\wE^n_h &:= \left\{S_{h,d_{\wE^{n-1}_h}}\big(d_{\wE^{n-1}_h}\big)< 0\right\}
\\
E^n_h &:= \left\{ S_{h,d_{E}}\big(d_{E^{n-1}_h}\big)< 0\right\}
\end{aligned}\end{equation}
for all $n\in\mathbb N$.
We also let $\E_h(t):=\wE^{[t/h]}_h$ and $E_h(t):=E^{[t/h]}_h$.
Notice that $E_h(t)$ is the discretized evolution corresponding to the mean curvature flow
with obstacle $\Om=E$ (see Definition \ref{defscheme}), while 
$\E_h(t)$ is an implicit scheme for \eqref{eqkpos}.

\begin{proposition}\label{prostacle}
Let $h>0$ and let $E\subset \R^n$ be a set with compact boundary. Then 
$$
\E_h(t)=E_h(t)\qquad {\ for\ all\ }\quad t\ge 0\,.
$$
In particular
\begin{equation}\label{comph}
E_h(t_2) \subseteq E_h(t_1)\qquad {\ for\ all\ }\quad t_1\le t_2.
\end{equation}
\end{proposition}

\begin{proof}
We have to show that $\E^n_h=E^n_h$ for all $n\in\mathbb N$.
By the definition we have $\E^1_h =E^1_h=:F$. If we also show that 
$\E^2_h =E^2_h$, then the thesis follows by iteration. 
As $d_F\ge d_E$, by Proposition \ref{promono} we have that $S_{h,d_E}(d_F)\ge S_{h,d_E}(d_E)$,
so that 
\begin{equation}\label{recall}
E^2_h=\left\{ S_{h,d_E}(d_F)< 0\right\}
\subset \left\{ S_{h,d_E}(d_E)< 0\right\}=F.
\end{equation}
By Proposition \ref{proset} we know that $E^2_h$ is the minimal solution of 
\[
\min_{X\subset E}P(X)+\frac 1 h \int_X d_F\,dx\,.
\]
Recalling \eqref{recall} it then follows that $E^2_h$ is also the minimal solution of 
\[
\min_{X\subset F}P(X)+\frac 1 h \int_X d_F\,dx
\]
and hence coincides with
$\E^2_h$, again by Proposition \ref{proset}.
\end{proof}

Proposition \ref{prostacle} implies that the evolution \eqref{eqkpos},
with initial set $E$, can be seen as a particular case of \eqref{formalevol} with $\Om=E$.
As a consequence, from Theorem \ref{thshort} we get  
a short time existence result for regular solutions to \eqref{eqkpos}.

\begin{corollary}\label{corpos}
Let $E\subset \R^2$ with compact boundary of class $C^{1,1}$. 
Then there exists $T>0$ such that \eqref{eqkpos} admits a unique solution 
$E(t)$ on $[0,T]$ with $E(0)=E$ and $\partial E(t)$ a compact set of class $C^{1,1}$
for all $t\in [0,T]$.
Moreover
\begin{equation}\label{comp}
E(t_2) \subseteq E(t_1)\qquad {\ for\ all\ }\quad t_1\le t_2.
\end{equation}
\end{corollary}

\begin{proof}
Thanks to Theorem \ref{thshort} 
there exist $T>0$ 
and a unique solution $E(t)$ of \eqref{formalevol} on $[0,T]$, with $E(0)=E=\Omega$ and $\partial E(t)$ of class $C^{1,1}$.
By Proposition \ref{prostacle}, for all $\bar t\in [0,T)$, $E(t)$ is the solution of \eqref{formalevol} on $[\bar t,T]$ 
with obstacle $\Om=E(\bar t)$.
In particular, letting as above $d(x,t)=d_{E(t)}(x)$ and recalling \eqref{eqdistsup}, this implies
\begin{equation*}
\frac{\partial d}{\partial t} = \max\left(\Delta d + O(d),0\right)\qquad {\rm a.e.\ in\ }U\times [0,T]\,,
\end{equation*}
that is, $E(t)$ is the solution of \eqref{eqkpos} in the sense of Definition \ref{defsubpos}.

The uniqueness of $E(t)$ follows from Lemma \ref{lemcon}, and \eqref{comp} follows from \eqref{comph}.
\end{proof}

\section*{Acknowledgements}
The authors wish to thank the anonymous referees for their careful
reading of the paper and helpful comments. Part of this work was
done during a visit of the second author to Padova university,
supported by the GNAMPA 2011 project {\it Propagazione di fronti e comportamento asintotico in connessione a problemi di omogeneizzazione}. The third author also acknowledges partial support by the Fondazione CaRiPaRo Project {\it Nonlinear Partial Differential Equations: models, analysis, and control-theoretic problems}.



\begin{thebibliography}{10}

\bibitem{Alm2009}
L. Almeida, P. Bagnerini, A. Habbal, S. Noselli, F. Serman.
\newblock Tissue repair modeling. 
\newblock in {\em Singularities in Nonlinear Evolution Phenomena and Applications},
CRM Series, vol. 9, 2009.

\bibitem{Alm2011}
L. Almeida, P. Bagnerini, A. Habbal, S. Noselli, F. Serman.
\newblock A mathematical model for dorsal closure.
\newblock {\em J. Theor. Biol.}, 268(1):105--119, 2011.

\bibitem{AlmBagHab}
L. Almeida, P. Bagnerini, A. Habbal.
\newblock Modeling actin cable contraction.
\newblock {\em Preprint}, 2011, to appear in {\em Computers
and Mathematics with Applications}.

\bibitem{AlmDem}
L. Almeida, J. Demongeot.
\newblock Predictive power of ``a minima'' models in biology.
\newblock {\em Preprint}, 2011, to appear in {\em Acta Biotheoretica}.

\bibitem{ATW}
F.~Almgren, J.~E. Taylor, and L.-H. Wang.
\newblock Curvature-driven flows: a variational approach.
\newblock {\em SIAM J. Control Optim.}, 31(2):387--438, 1993.

\bibitem{AmbrosioMM}
L.~Ambrosio.
\newblock Movimenti minimizzanti.
\newblock {\em Rend. Accad. Naz. Sci. XL Mem. Mat. Appl. (5)}, 19:191--246,
  1995.

\bibitem{AFP}
L.~Ambrosio, N.~Fusco, and D.~Pallara.
\newblock {\em Functions of bounded variation and free discontinuity problems}.
\newblock Oxford Mathematical Monographs. The Clarendon Press Oxford University
  Press, New York, 2000.

\bibitem{BCN}
G. Barles, A. Cesaroni and M. Novaga,
\newblock Homogenization of fronts in highly heterogeneous media.
\newblock {\em SIAM J. on Math. Anal.}, 43(1):212--227, 2011.

\bibitem{BDL}
G. Barles and F. Da Lio.
\newblock Remarks on the Dirichlet and State-Constraint Problems for Quasilinear Parabolic Equations.  
\newblock {\em Adv. Differential Equations}, 8(8):897--922, 2003.

\bibitem{BellettiniNovaga-B3}
G.~Bellettini and M.~Novaga.
\newblock Comparison results between minimal barriers and viscosity solutions
  for geometric evolutions.
\newblock {\em Ann. Scuola Norm. Sup. Pisa Cl. Sci.}, 26(1):97--131, 1998.

\bibitem{BCCN}
G.~Bellettini, V. Caselles, A.~Chambolle and M.~Novaga.
\newblock Crystalline mean curvature flow of convex sets.  
\newblock {\em Arch. Ration. Mech. Anal.}, 179(1):109--152, 2006.

\bibitem{Caf}
L. Caffarelli. 
\newblock The obstacle problem revisited. 
\newblock {\em J. Four. Anal. Appl.}, 4:383--402, 1998.

\bibitem{CLS}
P. Cardaliaguet, P.-L. Lions and P. Souganidis.
\newblock A discussion about the homogenization of moving interfaces.
\newblock {\em J. Math. Pures Appl.}, 91(4):339--363, 2009.

\bibitem{CasellesC-convex}
V. Caselles and A. Chambolle.
\newblock Anisotropic curvature-driven flow of convex sets.
\newblock {\em Nonlinear Anal.}, 65(8):1547--1577, 2006.

\bibitem{ChamMCM}
A.~Chambolle.
\newblock An algorithm for mean curvature motion.
\newblock {\em Interfaces Free Bound.}, 6(2):195--218, 2004.

\bibitem{ChFor}
A. Chambolle, V. Caselles, D. Cremers, M. Novaga and T. Pock,
\newblock An introduction to Total Variation for Image Analysis.
\newblock In {\em Theoretical Foundations and Numerical Methods for Sparse Recovery},
De Gruyter, Radon Series Comp. Appl. Math., vol. 9, pp. 263-340, 2010.

\bibitem{CN-ATW}
A.~Chambolle and M.~Novaga.
\newblock Implicit time discretization of the mean curvature flow with a
  discontinuous forcing term.
\newblock {\em Interfaces Free Bound.}, 10:283--300, 2008.

\bibitem{CB}
B. Craciun and K. Bhattacharya. 
\newblock Effective motion of a curvature-sensitive interface through a heterogeneous medium. 
\newblock {\em Interfaces Free Bound.}, 6:151--173, 2004.

\bibitem{users}
M. Crandall, H. Ishii and P.-L. Lions.
\newblock User's guide to viscosity solutions of second order partial differential equations.
\newblock {\em Bull. Amer. Math. Soc.}, 27(1):1--67, 1992.

\bibitem{DL}
F. Da Lio.
\newblock Comparison results for quasilinear equations in annular domains and applications.
\newblock {\em Comm. Partial Differential Equations}, 27(1-2):283--323, 2002.

\bibitem{DKY}
N. Dirr, G. Karali and N. K. Yip. 
\newblock Pulsating wave for mean curvature flow in inhomogeneous medium. 
\newblock {\em European J. Appl. Math.}, 19:661-699, 2008.

\bibitem{Hutson}
M. S. Hutson, Y. Tokutake, M. Chang, J. Bloor, S. Venakides, D. P Kiehart and G. Edwards.
\newblock Forces for morphogenesis investigated with laser microsurgery and quantitative modeling. 
\newblock {\em Science}, 300(5616):145--149, 2003.

\bibitem{KohnSerfaty}
R. V. Kohn, S.~Serfaty.
\newblock A deterministic-control based approach to fully nonlinear parabolic and elliptic equations.
\newblock {\em Comm. Pure Appl. Math.}, 63:1298--1350, 2010.

\bibitem{LuckhausSturz}
S.~Luckhaus and T.~Sturzenhecker.
\newblock Implicit time discretization for the mean curvature flow equation.
\newblock {\em Calc. Var. Partial Differential Equations}, 3(2):253--271, 1995.

\bibitem{Meyer}
Y.~Meyer.
\newblock {\em Oscillating patterns in image processing and nonlinear evolution
  equations}, volume~22 of {\em University Lecture Series}.
\newblock American Mathematical Society, Providence, RI, 2001.
\newblock The fifteenth Dean Jacqueline B. Lewis memorial lectures.

\bibitem{Mir}
M. Miranda.
\newblock Frontiere minimali con ostacoli. 
\newblock {\em Ann. Univ. Ferrara}, 16(1):29--37, 1971.

\bibitem{Phillips}
R. Phillips.
\newblock {\em Crystals, Defects and Microstructures.}
\newblock Cambridge University Press, 2001.

\bibitem{Sethian}
J. A. Sethian.
\newblock {\em Level Set Methods and Fast Marching Methods:
Evolving Interfaces in Computational Geometry, Fluid Mechanics,
Computer Vision and Materials Science.}
\newblock Cambridge University Press, 1999. 
\end{thebibliography}
\end{document}